\documentclass[review]{elsarticle}
\usepackage[T1]{fontenc}
\usepackage{amsmath , amsfonts , amssymb , mathrsfs , amsthm , mitpress}
\usepackage{graphicx}
\usepackage{geometry}
\usepackage{hyperref}

\newtheorem{theorem}{Theorem}[section]
\newtheorem{corollary}[theorem]{Corollary}
\newtheorem{definition}[theorem]{Definition}

\newtheorem{proposition}[theorem]{Proposition}
\newtheorem{remark}[theorem]{Remark}

\def\N{\mbox{I\hspace{-.15em}N}}
\def\R{\mbox{I\hspace{-.15em}R}}
\def\P{\mbox{I\hspace{-.15em}P}}

\begin{document}

\begin{frontmatter}

\title{Construction of a $C^2$ class finite element based on the Clough-Tocher subdivision }

\author[mymainaddress]{\textsc{Haudi\'{e}} Jean-St\'{e}phane Inkp\'{e} \corref{mycorrespondingauthor}}
\cortext[mycorrespondingauthor]{Corresponding author}
\ead{jsinkpe@gmail.com}
\author[mymainaddress]{\textsc{Koua Brou} Jean Claude}
\ead{k\_brou@hotmail.com}
\author[mysecondaryaddress]{A. \textsc{Le M\'ehaut\'{e}}}

\address[mymainaddress]{UFR de Math\'{e}matiques et Informatique, Universit\'{e} F\'{e}lix Houphou\"{e}t Boigny de Cocody 22 BP 582 Abidjan 22, C\^{o}te d'Ivoire}
\address[mysecondaryaddress]{Laboratoire Jean Leray, D\'{e}partement de Math\'{e}matiques,Universit\'{e} de Nantes 2, rue de la Houssini\`{e}re, F--44072, Nantes Cedex, FRANCE.}

\begin{abstract}
In this paper, we construct a $C^2$ finite element based on the Clough-Tocher subdivision. We use derivatives order up to two at the vertices and cross boundary derivatives order up to two along the exterior edges of the triangle. The centroid of the triangle is just evaluated. The interpolant used is globally $C^2,$ has local support, is piecewise polynomial of degree less or equal to 5.
\end{abstract}

\begin{keyword}
Clough-Tocher scheme \sep Piecewise polynomial \sep Finite element %

\MSC[2010] 65-D05 \sep 65-D07\sep 65-D10.
\end{keyword}

\end{frontmatter}


\section{Introduction}
Let $\Omega $\, be a bounded polygonal domain, and $\Delta $\, a partition of $\bar{\Omega}$\, made up of triangles. Let assume the mild assumption that this triangulation is of Delaunay type. Here, we focus on the case where each triangle of the collection $\Delta $\, is splitted into three sub-triangles by joining some interior points to each vertices. This triangulation, noted $\Delta_{CT} $\, is the so-called Clough-Tocher's triangulation, in reference to the pioneer authors who have considered this type of triangulation \citet{CT65}. Without loss of generalities, we take the centroid as the split point for each triangles. 
A bivariate spline function is then obtained by building a piecewise function on $\Delta_{CT} $\, which is polynomial with the same degree on each triangle on the same partition.
Several works has been done on the splines functions with respect to the triangulation of Clough-Tocher, one may especially refer to \citet{CT65,Alfeld1,AlLLS02,LaghSablon93}.
It should be noted that this type of approach, which involves cutting the triangle into sub-triangles has been used by several authors in the literature \citet{CT65,Alfeld1,AlLLS02,MjlLls07}. The scheme claim to yield a smooth piecewise polynomial of a low fixed degree.

Spline functions and finite elements approximation are close techniques. Constructing locally supported spline functions is quite easy when dealing with finite elements. In this case, the support of the spline is reduced to an element and its neighbors in the triangulation. Morever, the explicit construction can be made for each element individually.\\
Similary, finite elements are useful e.g. for the design of cars and aircraft. For this purpose, polynomial finite elements with $C^2$ smoothness of surfaces have been heavily studied.\\
Generally, the remaining B-coefficients are determined from the $ C^1 $\, and $ C^2 $\, conditions. Here we don't proceed in this way.\\
As in a previous paper~\citet{Kouab5} we were able to compute the B-coefficients of a piecewise cubic polynomial on $ \Delta_{CT} $\, with global $ C^1 $\, continuity by using uniquely subdivision algorithms and degree raising principle, it is natural to use the same approach for a piecewise quintic polynomial on $ \Delta_{CT} $\, with global $ C^2 $\, continuity. Moreover, the interpolant which is constructed, has local support (i.e. evaluation at a point in a specific triangle requires data only on that triangle and its neighbors in the triangulation).\\
We assume, we are given values and derivatives of order up to 2 at the vertices and cross boundary derivatives up to 2 along the exterior edges of the triangle~\citet{Alfeld1,Farin3}. Explicit formulas are given for the coefficients of this interpolant.

The paper is organized as follows. In the following section, we introduce some notations and recall some basic results, then we present our process for constructing a $C^2$ finite element.

\section{Preliminaries}

\subsection{Notations}

\begin{itemize}
\item $\left\vert \alpha\right\vert = \sum\limits_{i=1}^{n}\alpha_{i} = \alpha_{1}+\cdots+\alpha_{n} $
\item $ \alpha! = \prod\limits_{i=1}^{n}\alpha_{i}! = \alpha_{1}!\times\cdots\times\alpha_{n}! $
\item  $\lambda^{\alpha} = \prod\limits_{i=1}^{n} \lambda_{i}^{\alpha_{i}} =  \lambda_{1}^{\alpha_{1}} \times \cdots  \times \lambda_{n}^{\alpha_{n}} $
\item $\mathbb{P}_{d}\left( \mathcal{K} \right)$ is the space of bivariate polynomials of degree less or equal to $d $ defined on $\mathcal{K} $.
 \end{itemize}
\subsection{Definitions}
\paragraph{ }  Let $\mathcal{K}=\langle A_{1},A_{2},A_{3}\rangle $ be a triangle of vertices $A_{1},A_{2},A_{3} $ counter-clockwise oriented.
\paragraph{ }  Let $\lambda=\left(  \lambda_{1},\lambda_{2},\lambda_{3}\right)$ be a mulit-index of $\R$ and $\alpha=\left(  \alpha_{1},\alpha_{2},\alpha_{3}\right) $ a mulit-index of $\N $.
\begin{definition}[\textbf{Finite element}] %
Let us recall from~\citet{Ciarlet2} that a finite element is a $ \mathbf{triplet} \left(\mathcal{K},\mathcal{E},\mathcal{L}\right) $ %
which satisfies:
\begin{itemize}
\item $\mathcal{K}$ is a ${\R}^{2}$ convex polygon,
\item $\mathcal{E}$ is a vectorial space of functions defined on $\mathcal{K},$ %
\item $\mathcal{L}$ is an element of the dual of $\mathcal{E}$ formed by $s$ linear forms $\left(  l_{1},\ldots,l_{s}\right) $ defined on $\mathcal{E}$ such that
\end{itemize} %
$\mathcal{L} $ be $\mathcal{E}$-Unisolvent i.e the mapping
\begin{align*}
\Psi &  :\mathcal{E}\longrightarrow{\R}^{s}\\
\mathcal{P}  &  \mapsto\left(  l_{1}\left(  \mathcal{P}\right)  ,\ldots ,l_{s}\left(  \mathcal{P}\right)  \right)
\end{align*}
is an \textbf{isomorphism}. \\
In other words, $\mathcal{L}$ is $\mathcal{E}-$unisolvent $iff$ the two following conditions are satisfied:
\begin{enumerate}[(i)]
\item $card\mathcal{L}=\dim \mathcal{E} $,
\item if $\mathcal{L}$ being the set $ \left\{ l_{i} \right\}_{i=1,\ldots ,\dim \mathcal{E}}, $ given $v\in \mathcal{E}$ such that $ l_{i}\left( v\right) =0,$\ $i=1,\ldots ,\dim \mathcal{E},$ then \  $ v\equiv 0.$
\end{enumerate}
The linear forms $\left(  l_{1},\ldots,l_{s}\right)  $ are called \textbf{local degrees of freedom} and $\mathcal{L}=\left( l_{i}\right)  _{i=1\ldots s}$ is called \textbf{set of degrees of freedom} or \textbf{set of nodal values} of $\mathcal{K}$ .%
\end{definition}
\begin{definition}[\textbf{Barycentric coordinates }]
\textbf{The unique solution} $ \lambda $\ of the system :%
\begin{equation} \label{eqnbar}%
\left\{
\begin{array}
[l]{l}%
\mathcal{V} = \sum\limits_{i=1}^{3}\lambda_{i}A_{i}\\
\left\vert \lambda\right\vert =1
\end{array}
\right.
\end{equation}
is called \textbf{barycentric coordinates} of the point $\mathcal{V}$ with respect to $\mathcal{K}$.
\end{definition}

\begin{definition}[\textbf{Polynomial of Bernstein}]
The polynomial of degree $d$ defined on $\mathcal{K}$ by :%
\begin{align*}
B_{\alpha}^{d}\left(  \lambda\right) &= \frac{d!}{\alpha!}\lambda^{\alpha} \\
 &= \frac{d!}{\alpha_{1}!\alpha_{2}!\alpha_{3}!}\lambda_{1}^{\alpha_{1}} \lambda_{2}^{\alpha_{2}}\lambda_{3}^{\alpha_{3}}
\end{align*}
where $\left\vert \lambda\right\vert =1$\ and $\left\vert \alpha\right\vert
=d $ \ is called \textbf{Bernstein polynomial} of degree $d$ on $\mathcal{K}$.
\end{definition}
\begin{definition}[\textbf{BB-form}]
Every polynomial $\mathcal{P} \in\mathbb{P}_{d}\left(
\mathcal{K} \right)$  can be written in a unique way as
\begin{equation} \label{eqnbb}
\mathcal{P}\left(  \mathcal{V}\right)  =\sum\limits_{\left\vert \alpha \right\vert =d}b\left(  \alpha\right)  B_{\alpha}^{d}\left(  \lambda\right)
\end{equation} %
It is the \textbf{Bernstein-B\'{e}zier} form of $\mathcal{P} $ and the $ b(\alpha ) $ are the \textbf{B-coefficients or the ordinates of B\'{e}zier.}
\end{definition}
\begin{definition}[\textbf{Triangle of B\'{e}zier and control points}]
The set
\[
\left\lbrace (\xi_{\alpha }, b(\alpha )),\vert\alpha\vert =d\right\rbrace
\]
where
\begin{equation*}
\xi_{\alpha }=\frac{\alpha _{1}A_{1}+\alpha _{2}A_{2}+\alpha _{3}A_{3}}{d}
\end{equation*}
is called the set of \textbf{control points} and defines the \textbf{triangle of B\'{e}zier}.
\end{definition}

\subsection{Properties}
\begin{proposition}[\textbf{Cross derivatives}]
Let $\mathcal{P\in}\mathbb{P}_{d}\left(  \mathcal{K}\right) $ be such that
\begin{equation*}
\mathcal{P}\left(\mathcal{V}\right)  =\sum\limits_{\left\vert \alpha \right\vert =d}b\left(\alpha\right)B_{\alpha}^{d}\left(  \lambda \right)
\end{equation*}
Then
\begin{enumerate}[(1)]
\item \begin{equation*}
D_{\delta}^{r}\mathcal{P}\left( \mathcal{V}\right) = \sum\limits_{\left\vert \alpha\right\vert = d-r}\frac
{d!\times p_{\alpha}\left(  \delta\right)}{\left( d-r\right)!}  B_{\alpha}^{d-r}\left(  \lambda\right)
\end{equation*}
with
\begin{equation*}
p_{\alpha}\left(  \delta\right)  =\sum
\limits_{\left\vert \beta\right\vert =r}b\left(  \alpha+\beta\right)
B_{\beta}^{r}\left(  \delta\right)
\end{equation*}
\item \begin{equation*}
D_{\delta}^{r}D_{\eta}^{s}\mathcal{P}\left( \mathcal{V} \right) =
\sum\limits_{\left\vert \alpha\right\vert = d-r-s}
\frac{d! \times p_{\alpha}\left( \delta,\eta\right)}{\left(d-r-s
\right)!} \times B_{\alpha }^{d-r-s}\left( \lambda\right)
\end{equation*}
with
\begin{equation*}
p_{\alpha}\left( \delta,\eta\right)  = \sum\limits_{\left\vert \beta\right\vert =s}
p_{\alpha+\beta}\left(  \delta\right)  B_{\beta}^{s}\left(  \eta\right)
\end{equation*}
\end{enumerate}
\end{proposition}
\begin{corollary}[\textbf{Directional derivatives}] \label{derivdir}
Let $\left.  \delta=\overrightarrow{A_{k}A_{l}}\right. $ and $\left.
\eta=\overrightarrow{A_{k}A_{n}}\right.  $ two different directions and ,%
\begin{equation} \label{eqnmixte}%
D^{r+s}\mathcal{P}\left(  \mathcal{V}\right)  \cdot\left(  \delta^{r},\eta ^{s}\right) = \sum\limits_{\left\vert \alpha\right\vert =d-r-s} b_{\alpha}^{r,s}\left( \delta,\eta\right) B_{\alpha}^{d-r-s}\left( \lambda\right)
\end{equation}
\end{corollary}
with
\begin{equation*}
b_{\alpha}^{r,s}\left( \delta,\eta\right) = \frac{d!} {\left( d-r-s \right)!}\left( \Delta_{kl}^{r} \Delta_{kn}^{s}b\left( \alpha\right)  \right)
\end{equation*}
where $ \Delta_{ij}b\left(  \alpha\right)  =b\left(  \alpha
+\epsilon_{j}\right)  -b\left(  \alpha+\epsilon_{i}\right) $ and $\epsilon_{k}$ being the $k^{th}$ vector in the ${\R}^{3}$ %
canonical basis, $k=1,2,3$.

\begin{theorem}[\textbf{On $C^{r}$ Continuity}] \label{theocont}
Let $\mathcal{K}=\langle A,B,C\rangle $ and
$\widehat{\mathcal{K}} = \langle\widehat{C},B,A\rangle $
be two triangles sharing the same edge $ \left[  AB\right]. $ Given $\lambda $ and $\widehat{\lambda} $ the barycentric coordinates with respect to $ \mathcal{K} $ and $\widehat{\mathcal{K}}$ .\\
Let
$ \mathcal{P}\left( \mathcal{V}\right) = \sum \limits_{\left\vert\alpha\right\vert=d} b_{\alpha}B_{\alpha}^{d}\left( \lambda\right)$ \, and \, $ \widehat{ \mathcal{P}}\left( \mathcal{V}\right) =\sum \limits_{\left\vert \alpha\right\vert=d} \widehat{b}_{\alpha} B_{\alpha} ^{d}(\widehat{\lambda}) $ \; be two polynomials of degree  $d$ respectively defined on  $\mathcal{K} $ and  $ \widehat{ \mathcal{K}} $.
\paragraph{ }  The \textbf{$C^{r}$ continuity of $\mathcal{P}$ and $\widehat{ \mathcal{P}}$} with respect to the edge $\left[  AB\right]  $ is satisfied if and only if for $0\leq k\leq r$ and $0\leq\rho\leq d-k$%
\begin{equation} \label{eqncont}%
\widehat{b}_{k,\rho,d-k-\rho}=b_{d-k-\rho,\rho,0}^{k}\left(  \alpha\right)
\end{equation}
where $\left.  b_{\mu}^{k}\left(  \alpha\right)  \right.  $ is defined by the following recurrence formula :
For $k\geq1$ and $\left\vert \mu\right\vert =d-k$
\[
\left\{
\begin{array}[l]{l}%
b_{\mu}^{0}\left(  \alpha\right)  =b_{\mu}\\
b_{\mu}^{k}\left(  \alpha\right)  =\sum\limits_{i=1}^{3}\alpha_{i}%
b_{\mu+\epsilon_{i}}^{k-1}\left(  \alpha\right)
\end{array}
\right.
\]
$ \alpha=\widehat{\lambda}(\widehat{C)}$ being the barycentric coordinates of $\widehat{C}$ in $\mathcal{K}$ and $\epsilon_{k}$ the $k^{th}$ vector in the ${\R}^{3}$ canonical basis, %
$ k~=1,2,3$.~\cite{Farin4}.
\end{theorem}
\section{\texorpdfstring{$C^{2} $}{Lg} class finite elements}

\subsection{Expression of B\'{e}zier coefficients with respect to initial conditions}
\paragraph{ }  In the remainder of the paper, we consider a triangle
$ \mathcal{K}=\langle A_{1},A_{2},A_{3}\rangle $  which is split into 3 subtriangles $ \mathcal{K}_{i}=\langle A_{i+2},A_{0},A_{i+1} \rangle $ from the centroid $A_{0}$ and we will use the \textbf{modulo 3} congruence for $i>0 $.  \\ %
On $\mathcal{K} $, we define the following set of degrees of freedom : %
\begin{eqnarray} \label{eqndeglib}%
\sum\nolimits_{\mathcal{K}_{i}}^{2} &=& \lbrace f\left(  A_{0}\right)  ;\partial^{r}f\left( A_{i0}\right)  ,\partial^{r}f\left(  A_{i4} \right),r\leq2; \\
&& D^{2}f \left( A_{i1}\right)\cdot{
\overrightarrow{ A_{0} A_{i4} } }^{2},
D^{2}f\left( A_{i3}\right) \cdot{\overrightarrow{A_{0}A_{i0}}}^{2},\notag \\
&& Df\left( A_{i2}\right)\cdot\overrightarrow{ A_{0} A_{i2}} \rbrace \notag
\end{eqnarray}
with $ A_{ij}=\left( 1-\frac{j}{4}\right) A_{i+1}+\frac{j}{4}
A_{i+2} $ for $j=0,1,2,3,4 $ and for $i=1,2,3. $ \\
Let us consider the following space of polynomial functions :
\begin{equation} \label{eqnespfonc}%
\mathbb{PP}_{5}\left(  \mathcal{K}\right)  \mathbb{=}\left\{ S\in
C^{0}\left(  \mathcal{K}\right) : S|_{\mathcal{K}_{i}} \in \mathbb{P}_{5}, \, i=1,2,3\right\}
\end{equation}
The triplet $\left(  \mathcal{K}\text{\quotesinglbase } \mathbb{PP}_{5}\left( \mathcal{K}\right)  \text{\quotesinglbase }%
{\textstyle\sum\nolimits_{\mathcal{K}}^{2}} \right)  $\, where \, %
$\textstyle\sum\nolimits_{\mathcal{K}}^{2}= \bigcup\limits_{i=1}^{3} \sum\nolimits_{\mathcal{K}_{i}}^{2} $ \, is a finite element by construction. It consists on 16 local degrees of freedom on each of the subtriangle $\mathcal{K}_{i}$ as depicted on Figure~\ref{figure:fig1}.%
\begin{figure} [!htbp] \centering
\includegraphics[scale=0.25]{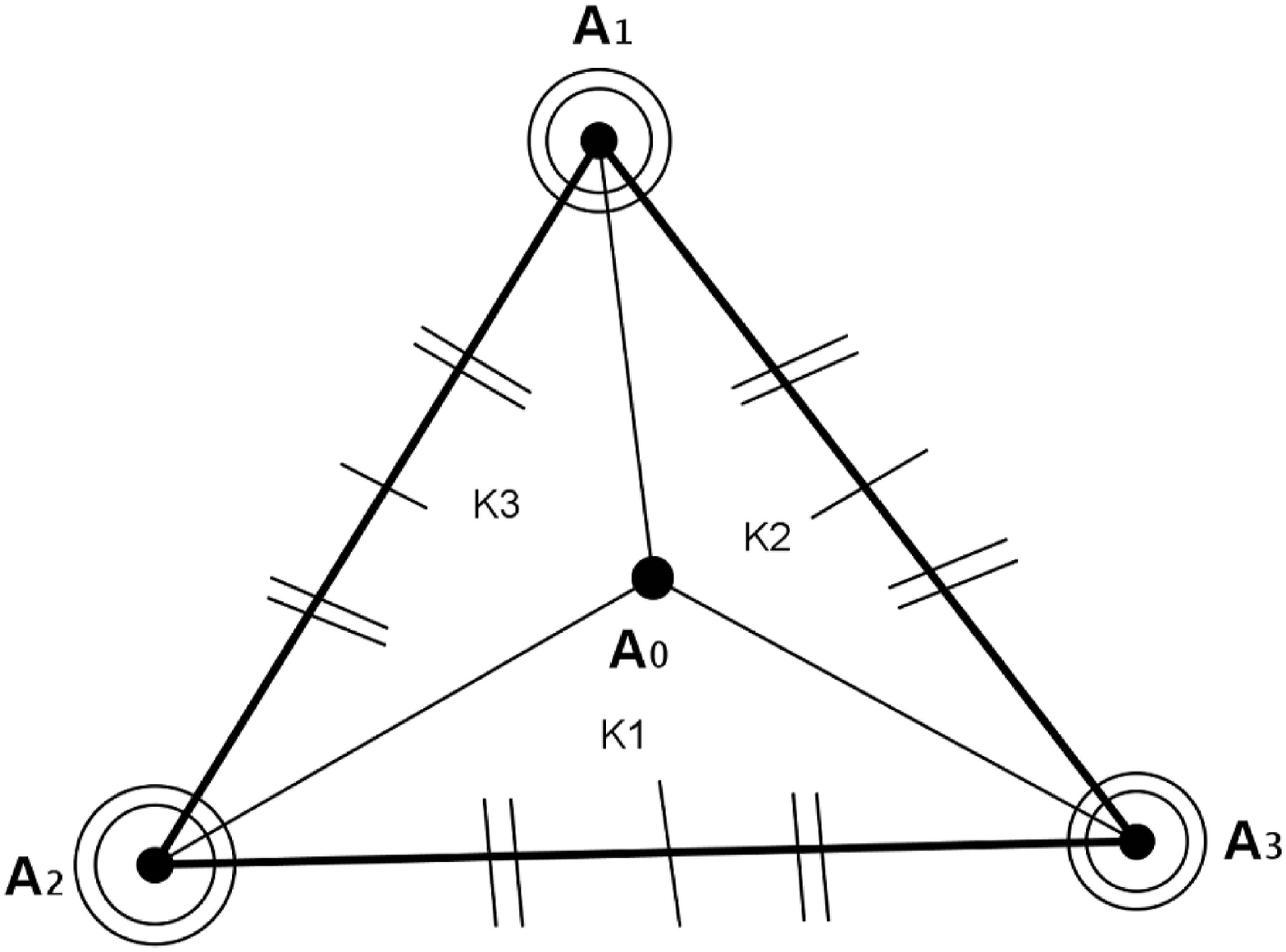}
\caption{16 degrees of freedom on $\mathcal{K}_{i} ,i=1,2,3.$}
\label{figure:fig1}
\end{figure}
\paragraph{ }  We notice that on each subtriangle, the 16 degrees of freedom are not sufficient to characterize a polynomial of degree 5. In fact, $ \dim\left(  \mathbb{P}_{5}\right)  =21 $ while $card\left( {\textstyle\sum\nolimits_{\mathcal{K}_{i}}^{2}}\right) =16 $. %
So it remains 5 indetermined coefficients on each of the subtriangle.
\paragraph{ } For a sufficiently regular function $ f \in \mathbb{PP}_{5}\left( \mathcal{K}\right)\cap C^{2}\left( \mathcal{K} \right), $ let define for $ i=1,2,3, $ \; $ \mathcal{P}_{i} \in \mathbb{P}_{5}\left(\mathcal{K}_{i}\right) $ such that %
\begin{equation} \label{eqnbbi}%
\mathcal{P}_{i}\left( \mathcal{V}\right)  =\sum\limits_{\left\vert
\beta\right\vert =5}C_{i}\left( \beta \right)  B_{\beta}^{5}\left(
\lambda_{i} \right)
\end{equation}
The 16 known B-coefficients on each $\mathcal{K}_{i} $ are given as follow:
\begin{equation*}
\left.
\begin{array}{r c l}
C_{i}\left(  5,0,0\right)   &  =& f\left(  A_{i4}\right),  \\
\\
C_{i}\left(  4,1,0\right)   &  =& \frac{1}{5}Df\left( A_{i4}\right)
\cdot \overrightarrow{A_{i4}A_{0}} + f\left( A_{i4}\right),  \\
\\
C_{i}\left(  4,0,1\right)   &  =& \frac{1}{5}Df\left( A_{i4}\right)
\cdot \overrightarrow{A_{i4}A_{i0}} + f\left( A_{i4}\right),  \\
\\
C_{i}\left(  3,2,0\right)   &  =& \frac{1}{20}D^{2}f\left(A_{i4} \right) \cdot {\overrightarrow{A_{i4}A_{0}}}^{2} + \frac{2}{5}Df\left(  A_{i4}\right)\cdot \overrightarrow{A_{i4}A_{0}}+ f\left(A_{i4} \right),\\
\\
C_{i}\left(  3,0,2\right)   &  =& \frac{1}{20}D^{2}f\left(A_{i4} \right) \cdot {\overrightarrow{A_{i4}A_{i0}}}^{2}+ \frac{2}{5}Df\left(  A_{i4}\right)\cdot \overrightarrow{A_{i4}A_{i0}}+ f\left(A_{i4} \right), 
\end{array}
\right.
\end{equation*}
\begin{equation*}
\left.
\begin{array}{r c l}
C_{i}\left(  3,1,1\right)   &  =& \frac{1}{20}D^{2}f\left(A_{i4} \right) \cdot \left( \overrightarrow{A_{i4}A_{0}},\overrightarrow{A_{i4}A_{i0}} \right) + \frac{1}{5}Df\left( A_{i4} \right) \cdot \overrightarrow{A_{i4}A_{0}} \\
  &  &  + \frac{1}{5} Df\left(  A_{i4} \right) \cdot \overrightarrow{A_{i4}A_{i0}} + f\left(A_{i4} \right) \\
  \\ 
C_{i}\left(  0,0,5\right)   &  =& f\left(  A_{i0}\right),  \\
\\
C_{i}\left( 0,1,4 \right)   &  =& \frac{1}{5}Df\left( A_{i0} \right) \cdot \overrightarrow{A_{i0}A_{0}} + f\left( A_{i0}\right),  \\
\\
C_{i}\left( 1,0,4 \right)   &  =& \frac{1}{5}Df\left( A_{i0} \right) \cdot \overrightarrow{A_{i0}A_{i4}} + f\left( A_{i0}\right),  \\
\\
C_{i}\left( 0,2,3 \right)   &  =& \frac{1}{20}D^{2}f\left(A_{i0} \right) \cdot {\overrightarrow{A_{i0}A_{0}}}^{2} + \frac{2}{5}Df\left(  A_{i0}\right)\cdot \overrightarrow{A_{i0}A_{0}}+ f\left(A_{i0} \right),  \\
\\
C_{i}\left( 2,0,3 \right)   &  =& \frac{1}{20}D^{2}f\left(A_{i0} \right) \cdot {\overrightarrow{A_{i0}A_{i4}}}^{2}+ \frac{2}{5}Df\left(  A_{i0}\right)\cdot \overrightarrow{A_{i0}A_{i4}}+ f\left(A_{i0} \right),  \\
\\
C_{i}\left( 1,1,3 \right)   &  =& \frac{1}{20}D^{2}f\left(A_{i0} \right) \cdot \left( \overrightarrow{A_{i0}A_{0}},\overrightarrow{A_{i0}A_{i4}} \right) + \frac{1}{5}Df\left( A_{i0} \right) \cdot \overrightarrow{A_{i0}A_{4}} \\
  &  &  + \frac{1}{5} Df\left(  A_{i0} \right) \cdot \overrightarrow{A_{i0}A_{0}} + f\left(A_{i0} \right)\\
\\
C_{i}\left( 2,1,2 \right)   &  =& \frac{8}{15}Df\left(A_{i2} \right) \cdot \overrightarrow{A_{i2}A_{0}} + \frac{1}{12}f\left(A_{i4}\right)+ \frac{1}{12}f\left(A_{i0}\right) \\
&  & + \frac{5}{12}C_{i}\left( 4,0,1 \right)+ \frac{5}{12}C_{i} \left( 1,0,4 \right)- \dfrac{1}{6}C_{i}\left(4,1,0\right)- \frac{1}{6} C_{i}\left( 0,1,4 \right),  \\
&  & - \frac{2}{3}C_{i}\left( 3,1,1 \right)- \dfrac{2}{3}C_{i} \left( 1,1,3 \right)+ \frac{5}{6}C_{i}\left(3,0,2 \right)+ \frac{5}{6} C_{i} \left( 2,0,3 \right)\\
C_{i}\left( 1,2,2 \right) &  =& -\frac{2}{45} D^{2}f \left(A_{i1} \right) \cdot {\overrightarrow{A_{0}A_{i4}}}^{2} -\frac{4}{15}Df\left(A_{i2} \right) \cdot \overrightarrow{A_{0}A_{i2}} -\frac{13}{12}f\left( A_{i0} \right) \\
&  & + \frac{2}{15}D^{2}f\left(A_{i3} \right) \cdot {\overrightarrow{A_{0}A_{i0}}}^{2} + \frac{5}{12}f \left(A_{i4} \right)- \frac{5}{6}C_{i}\left( 4,1,0 \right) \\
&  & + \frac{7}{12} C_{i} \left( 4,0,1 \right) + \frac{1}{3} C_{i} \left( 3,2,0 \right)- C_{i}\left( 3,1,1 \right)  \\
&  &  + \frac{1}{2}C_{i} \left( 3,0,2 \right) + \frac{1}{18}C_{i} \left( 2,0,3 \right) + \frac{17}{9}C_{i}\left( 1,1,3 \right)  \\
&  & -\frac{11}{12}C_{i} \left(1,0,4 \right) - \frac{10}{9}C_{i} \left(0,2,3\right) + \frac{13}{6} C_{i} \left( 0,1,4 \right)\\
C_{i}\left( 2,2,1 \right) &  =& -\frac{2}{45} D^{2}f \left(A_{i3} \right) \cdot {\overrightarrow{A_{0}A_{i0}}}^{2} -\frac{4}{15}Df \left(A_{i2}\right) \cdot \overrightarrow{A_{0}A_{i2}}-\frac{13}{12}f \left( A_{i4} \right) \\
&  & + \frac{2}{15}D^{2}f\left(A_{i1} \right) \cdot {\overrightarrow{A_{0}A_{i4}}}^{2} + \frac{5}{12}f \left(A_{i0} \right)- \dfrac{5}{6}C_{i}\left( 0,1,4 \right) \\
&  & + \frac{1}{3} C_{i} \left( 0,2,3 \right) + \frac{7}{12} C_{i} \left( 1,0,4 \right) - C_{i}\left( 1,1,3 \right)  \\
&  &  + \frac{1}{2}C_{i} \left( 2,0,3 \right) + \frac{1}{18}C_{i} \left( 3,0,2 \right) + \frac{17}{9}C_{i}\left( 3,1,1 \right)  \\
&  & - \frac{10}{9}C_{i} \left(3,2,0\right) - \frac{11}{12}C_{i} \left(4,0,1 \right) + \frac{13}{6} C_{i} \left( 4,1,0 \right) \\
C_{i}\left( 0,5,0 \right) &  =& f \left( A_{0} \right)
\end{array}
\right.
\end{equation*}
Then, the 5 B-coefficients to determine on each subtriangle
$\mathcal{K}_{i}$ are the following :%
\[
C_{i}\left(  0,3,2\right) \text{, }C_{i}\left(  0,4,1\right) \text{, }
C_{i}\left(  1,3,1\right)  \text{, }C_{i}\left(  1,4,0\right)  \text{, } C_{i}\left(  2,3,0\right)
\]
\paragraph{ } Finally, the polynomial $\left.  \mathcal{P\in}\text{ }\mathbb{PP}_{5}\left( \mathcal{K}\right) \right. $ has 37 out of 46 coefficients known. The 9 indetermined coefficients taking into account the $C^{0}$ continuity along the edge $ \left[ A_{0} A_{i} \right]  $ for $ i= 1,2,3, $ are represented on Figure~\ref{figure:fig2} below by crosses. The full circles come from conditions on  derivatives at the summits and the hollow circles are obtained from conditions on directional derivatives with respect to non-colinear directions on the exterior boundaries of the triangle $\mathcal{K} $.%
\begin{figure} [!htbp] \centering
\includegraphics[scale=0.18]{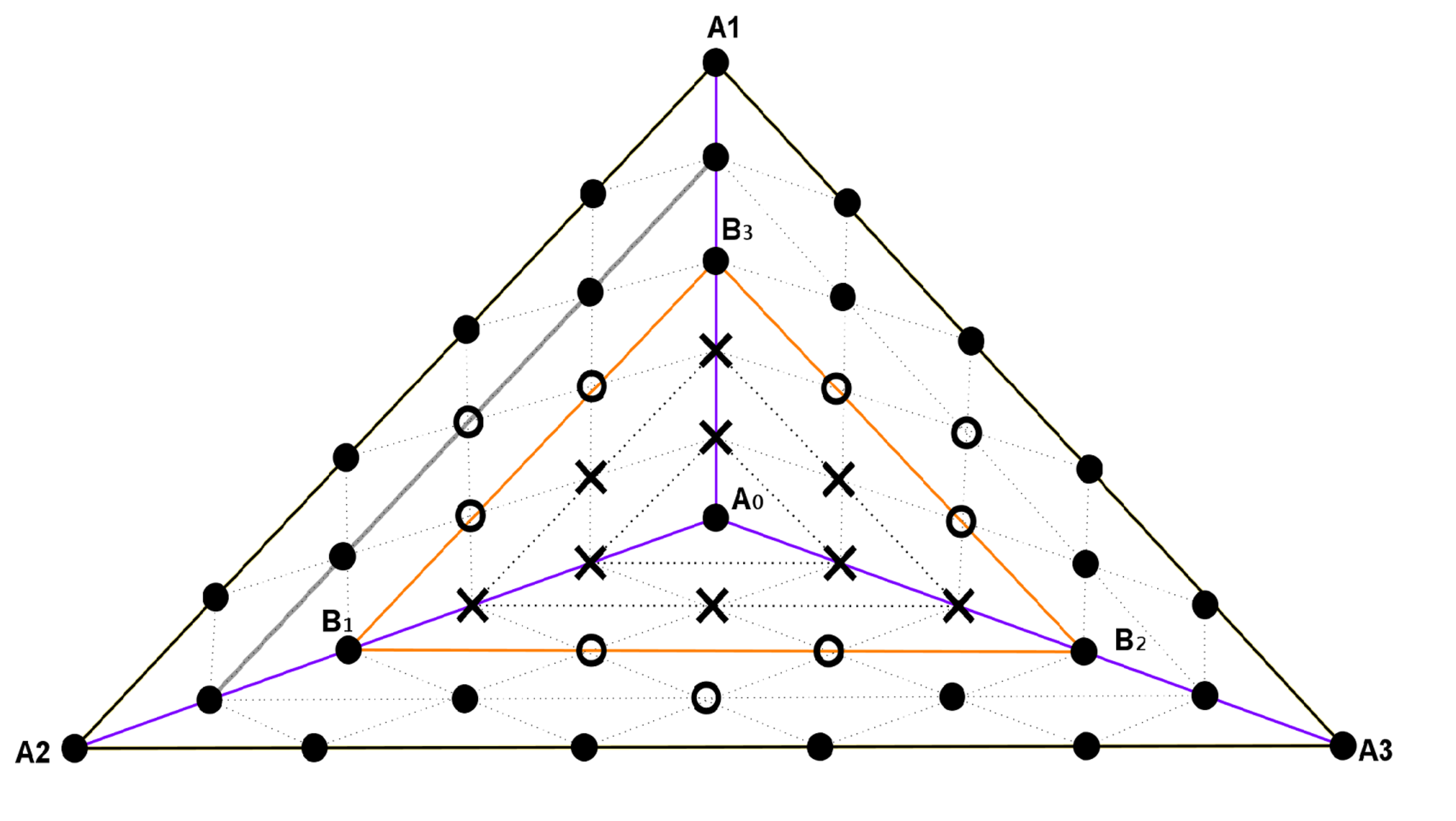}
\caption{Associated B\'{e}zier coefficients}
\label{figure:fig2}
\end{figure}
\subsection{Determination of the unknown coefficients of B\'{e}zier}
\paragraph{ } Now we propose a computation method of the five (05) B-coefficients remaining on each sub triangle $\mathcal{K}_{i}$.
\paragraph{ } To this end, let consider the set $\left. \widetilde{ \mathcal{K}}_{0} = \widetilde{\mathcal{K}} \cup\left\{ A_{0}\right\} \right. $ where $ \widetilde{\mathcal{K}}=\langle B_{1},B_{2},B_{3} \rangle $ with vertices $ B_{i}=\frac{2}{5}A_{0}+\frac{3}{5} A_{i+1}$ \\ %
for $i=1,2,3 $ and $\left. \mathbb{P}_{3}\left(  \widetilde {\mathcal{K}}_{0}\right) \right. $ the space of polynomial functions of degree less or equal to 3 defined on $\widetilde {\mathcal{K}}_{0}$. Let consider the set of degrees of freedom: %
\[
{\sum\nolimits_{\widetilde{\mathcal{K}}_{0}}}=\left\{f \left( A_{0} \right) \right\} \cup \left\{ f\left( B_{iij}\right) \right\}_{1\leq i,j \leq3}
\]%
where $ B_{iij}=\frac{2}{3}B_{i}+\frac{1}{3}B_{j} $. \\
The triplet $\left(  \widetilde{\mathcal{K}}_{0}, \mathbb{P}_{3}
(\widetilde{\mathcal{K}}_{0}), {\textstyle \sum \nolimits_{\widetilde{\mathcal{K}}_{0}}} \right)  $ is a Lagrange's finite element of type $3. $ ~\cite{Raviart6}.
\paragraph{ } For $ i>0 $ such that $ i\equiv 0 \pmod{3} $ , let define $ \widetilde{\mathcal{K}}_{i}=\langle B_{i+2} , A_{0} , B_{i+1} \rangle  $ a subdivision of $\widetilde{\mathcal{K}}$ into 3 subtriangles around the centroid $ A_{0} $. %
\begin{proposition}[\textbf{Algorithm of subdivision 1}]
Let $\sigma$\ be a 3-cycle, $ \gamma_{0}$\ the barycentric coordinates of $A_{0}$\ in $\widetilde{\mathcal{K}} $ and let define $\widetilde{\mathcal{P}} \in \mathbb{P}_{3}(\widetilde{\mathcal{K}}_{0})$ such that
\begin{equation} \label{eqnpol1}%
\widetilde{\mathcal{P}}\left(  \mathcal{V}\right)  = \sum\limits_{\left\vert \alpha\right\vert =3}\widetilde{b}\left( \alpha\right)B_{\alpha}^{3}\left( \gamma\right)
\end{equation}
with $\gamma=\left( \gamma_{1},\gamma_{2},\gamma_{3}\right) $ the barycentric coordinates in $\widetilde{\mathcal{K}}$. \\%
For $\left.  i=1,2,3, \right. $\, let define $\left.
\widetilde{\mathcal{P}}_{i}\in\mathbb{P}_{3}\right. $ such that $\left. \widetilde{\mathcal{P}}_{i}=\widetilde{\mathcal{P}}|_{\widetilde{\mathcal{K} }_{i}}\right.  $\ and $\rho=\left(  \mathcal{\rho}_{1},\mathcal{\rho}_{2},\mathcal{\rho}_{3}\right)  \in{\N}^{3}$ with $\left\vert \rho\right\vert =3 $.

If %
\begin{equation} \label{eqnpol2}%
\widetilde{\mathcal{P}}_{i}\left(  \mathcal{V}\right)  =\sum
\limits_{\left\vert \rho\right\vert =3}\widetilde{b}_{i}\left(  \rho\right)B_{\rho}^{3}\left(  \gamma_{i}\right)
\end{equation}
with $\gamma_{i}=\left(  \gamma_{i1},\gamma_{i2},\gamma_{i3}\right)  $\, the barycentric coordinates in $\widetilde{\mathcal{K}}_{i} $,\\ %
then %
\begin{equation} \label{eqnpol3}%
\widetilde{b}_{i}\left(  \rho\right)  =\sum\limits_{\left\vert \alpha
\right\vert =\rho_{2}}\widetilde{b}\left(  \alpha+\sigma^{i}\left(  \rho^{0}\right) \right)B_{\alpha}^{\rho_{2}}\left( \gamma_{0}\right)
\end{equation}
where $\rho^{0}=\left(\mathcal{\rho}_{3},\mathcal{\rho}_{1},0\right)$.
\end{proposition}%
\paragraph{ } It is sufficient to prove one of the cases due to the symetric property. Let us consider case $ i=3 $. %
\begin{proof}
$\left.  \forall\mathcal{V} \in \widetilde{\mathcal{K}}_{3}\right. $%
\begin{align*}
\mathcal{V}  &  \mathcal{=}\gamma_{31}B_{2}+\gamma_{32}A_{0}+\gamma
_{33}B_{1}\\
&  =\gamma_{31}B_{2}+\frac{1}{3}\gamma_{32}(B_{3}+B_{1}+B_{2}%
)+\gamma_{33}B_{1}\\
&  =(\gamma_{33}+\frac{1}{3}\gamma_{32})B_{1}+(\gamma_{31}+\frac{1}{3}\gamma_{32})B_{2}+\frac{1}{3}\gamma_{32}B_{3}\\
&  =\gamma_{1}B_{1}+\gamma_{2}B_{2}+\gamma_{3}B_{3}
\end{align*}
Let define $\left.  \nu=\left(  \delta,\beta,\xi\right)  \right.  $ a mulit-index of ${\N}$ such that $\left\vert \nu\right\vert =3 $
\begin{align*}
B_{\nu}^{3}\left(  \gamma\right) &  =\frac{3!}{\nu!}\gamma^{\nu}\\
&  =\frac{3!}{\delta!\beta!\xi!}\left(  \gamma_{33}+\frac{1}{3} \gamma_{32}\right)^{\delta}\left(\gamma_{31}+\frac{1}{3} \gamma_{32} \right)^{\beta}\left(\frac{1}{3}\gamma_{32}\right)^{\xi}\\
&  ={\sum\limits_{\mu=0}^{\beta}}{\sum\limits_{\eta=0}^{\delta}}
\frac{(3-\mu-\eta)!}{(\delta-\eta)!(\beta-\mu)!\xi!}\left(\frac{1}{3} \right)^{3-\mu-\eta}\\
& \times \frac{3!}{\mu!(3-\mu-\eta)!\eta!}\gamma_{31}^{\mu} \times \gamma_{32}^{3-\mu-\eta}\times\gamma_{33}^{\eta}\\
&  ={\sum\limits_{\mu=0}^{\beta}}{\sum\limits_{\eta=0}^{\delta}} B_{\delta-\eta,\beta-\mu,\xi}^{3-\mu-\eta}(\frac{1}{3},\frac{1}{3},\frac{1}{3}) B_{\mu,3-\mu-\eta,\eta}^{3}\left( \gamma_{3}\right)
\end{align*}
As $ \gamma_{0}=(\frac{1}{3},\frac{1}{3},\frac{1}{3}) $ and by posing $ {\rho}_{1}=\mu , {\rho}_{2}=3-\mu-\eta $ and ${\rho}_{3}=\eta\ $
we have $\left\vert \rho\right\vert =3 $
thus
\[
B_{\nu}^{3}\left( \gamma\right) =\sum\limits_{\left\vert \rho \right \vert =3}B_{\delta-\rho_{3},\beta-\rho_{1},\xi}^{\rho_{2}}(\gamma_{0}) \times B_{\rho}^{3}\left(  \gamma_{3}\right)
\]
On the other hand,
\begin{align*}
\widetilde{\mathcal{P}}\left(  \mathcal{V}\right)   &  =\sum
\limits_{\left\vert \nu\right\vert =3}\widetilde{b}\left(  \nu\right)  B_{\nu }^{3}\left(  \gamma\right) \\
&  =\sum\limits_{\left\vert \nu\right\vert =3}\widetilde{b}\left(  \nu\right)\sum\limits_{\left\vert \rho\right\vert =3} B_{\delta-\rho_{3},\beta-\rho_{1},\xi}^{\rho_{2}}(\gamma_{0})\times B_{\rho}^{3}\left(  \gamma_{3}\right)\\
&  =\sum\limits_{\left\vert \rho\right\vert =3}\left(  \sum \limits_{\left\vert\nu-\sigma^{3}(\rho^{0})\right\vert = \rho_{2}} \widetilde{b}\left(\nu\right)B_{\delta-\rho_{3},\beta-\rho_{1},\xi}^{\rho_{2}}(\gamma_{0})\right)B_{\rho}^{3}\left(  \gamma_{3} \right)
\end{align*}
by posing $\left. \alpha=\nu-\sigma^{3}\left(\rho^{0}\right)\right. $,\ we have $\left\vert \alpha\right\vert =\mathcal{\rho}_{2}.$\ Consequently%
\begin{align*}
\widetilde{\mathcal{P}}\left(  \mathcal{V}\right)   &  =\sum
\limits_{\left\vert \rho\right\vert =3}\left(  \sum \limits_{\left\vert \alpha\right\vert =\rho_{2}}\widetilde{b}\left(  \alpha+\sigma^{3}\left(\rho^{0}\right)  \right)  B_{\alpha}^{\rho_{2}}(\gamma_{0})\right)B_{\rho}^{3}\left(  \gamma_{3}\right) \\
&  =\sum\limits_{\left\vert \rho\right\vert =3} \widetilde{b}_{3} \left( \rho\right)B_{\rho}^{3}\left( \gamma_{3}\right) \\
&  =\widetilde{\mathcal{P}}_{3}\left( \mathcal{V}\right)
\end{align*}
with
\[
\widetilde{b}_{3}\left(  \rho\right)  =\sum\limits_{\left\vert \alpha
\right\vert =\rho_{2}}\widetilde{b}(\alpha+\sigma^{3}(\rho^{0})) B_{\alpha}^{\rho_{2}}(\gamma_{0})
\]
\end{proof}

\begin{remark}
Knowing the indexes $\left.  \alpha\right.  $\ of $\left. \widetilde{b}(\alpha)\right.  $ computed from $\left.
\widetilde{b}_{i}(\rho)\right.  $, it suffices to make a permutation of the latest in order to obtain the indexes $\left.  \beta\right.  $\ of $\left. \widetilde{b}(\beta)\right.  $\ computed in $\left.  \widetilde{b}_{i+1}(\rho)\right.  $ with modulo 3 congruence.
\end{remark}
\paragraph{ }  For $\left.  i=1,2,3\right.  $\ , there exists $
\widetilde{\mathcal{P}}_{i}^{(2)}\in\mathbb{P}_{5} $ such that $\widetilde{\mathcal{P}}_{i}^{(2)} = \widetilde{ \mathcal{P}}_{i}|_{\widetilde{\mathcal{K}}_{i}} $.  We just have to increase the degree of $ \widetilde{ \mathcal{P}}_{i} $.
\begin{proposition}[\textbf{Degree raising}]
For $  i=1,2,3  $, let define $ \widetilde{\mathcal{P}}_{i} \in \mathbb{P}_{3}\left( \mathcal{K}_{i}\right) $\ such that
\[
\widetilde{\mathcal{P}}_{i}\left( \mathcal{V}\right)  = \sum \limits_{\left\vert\alpha\right\vert =3}\widetilde{b}_{i}\left(  \alpha\right)B_{\alpha}^{3}\left(\gamma_{i}\right)
\]
then
\begin{equation} \label{eqnelev1}%
\widetilde{\mathcal{P}}_{i}\left(  \mathcal{V}\right)  =\sum\limits_{\left\vert\mu\right\vert= 4}\widetilde{b}_{i}^{(1)} \left( \mu\right)B_{\mu}^{4}\left( \gamma_{i} \right)
\end{equation}
where %
\begin{equation*} \label{eqnelevcoef1}%
\widetilde{b}_{i}^{(1)}\left(\mu\right) =\frac{1}{4}{\sum\limits_{k=1}^{3}}\mu_{k}\widetilde{b}_{i}\left( \mu-\epsilon_{k} \right)
\end{equation*}
with $\left\vert \mu\right\vert =4 $\ ,$ \epsilon_{k} $ the $ k^{th}$ \, vector in the ${\R}^{3}$ canonical basis and $\gamma_{i} = \left( \gamma_{i1},\gamma_{i2},\gamma_{i3}\right) $ represents the barycentric coordinates of $\mathcal{V}$ on $ \widetilde{\mathcal{K}}_{i} $.
\end{proposition} %
\begin{proof}
\[
\widetilde{\mathcal{P}}_{i}\left(\mathcal{V}\right) =\sum \limits_{\left\vert\alpha\right\vert =3}\widetilde{b}_{i}\left(  \alpha\right)B_{\alpha}^{3}\left(\gamma_{i}\right) \quad\text{As}\quad 1=\gamma_{i1}+\gamma_{i2}+\gamma_{i3}
\]
by doing a memberwise mutiplication and using the fact that for $\left.  i=1,2,3\right.  $
\[
\gamma_{ik}B_{\alpha}^{3}\left( \gamma_{i}\right) =\frac{\mu_{k}}{4}B_{\mu}^{4}\left( \gamma_{i}\right)
\]
where $\left.  \mu=\alpha+\epsilon_{k}\right. $, we have :%
\begin{align*}
\widetilde{\mathcal{P}}_{i}\left(  \mathcal{V}\right)   &  = \sum \limits_{\left\vert \mu\right\vert =4}\frac{1}{4}\left( \sum \limits_{k=1}^{3} {\mu}_{k}\widetilde{b}_{i}\left(  \mu-\epsilon_{k} \right) \right)  B_{\mu}^{4}\left(  \gamma_{i}\right) \\
&  =\sum\limits_{\left\vert \mu\right\vert =4}\widetilde{b}_{i}^{(1)}%
(\mu)B_{\mu}^{4}\left( \gamma_{i}\right)
\end{align*}
by posing for $\left\vert \mu\right\vert =4$
\[
\widetilde{b}_{i}^{(1)}(\mu)=\frac{1}{4}{\sum\limits_{k=1}^{3}} \mu_{k}\widetilde{b}_{i}\left( \mu-\epsilon_{k} \right)
\]
\end{proof}
\paragraph{ }  By applying a double increasing of the degree, we have:
\begin{equation} \label{eqnelev2}%
\widetilde{\mathcal{P}}_{i}\left( \mathcal{V}\right)  =\sum \limits_{\left\vert\mu\right\vert =5}\widetilde{b}_{i}^{(2)}(\mu)B_{\mu}^{5}\left(\gamma_{i} \right)
\end{equation}
with
\begin{equation} \label{eqnelevcoef2}%
\widetilde{b}_{i}^{(2)}(\mu) = \sum\limits_{k=1}^{3}
\sum\limits_{l=1}^{3} \frac{\mu_{k}\left( \mu_{l}-\delta_{kl}\right)}{20} \widetilde{b}_{i} \left(\mu - \epsilon_{k} -\epsilon_{l} \right)
\end{equation}
where $\left.  \delta_{kl}\right. $ is the Kronecker symbol.
\paragraph{ } For $ i=1,2,3,  $ let $\mathcal{P}_{i}$\ be a polynomial of degree less or equal to 5 defined on $\mathcal{K}_{i}$ and having for Bezier coefficients $\left\{C_{i}(\beta),\left\vert \beta\right\vert =5\right\} $. By making it coincide on $\widetilde{\mathcal{K} }_{i+2}$ with the polynomial $\widetilde{\mathcal{P}}_{i+2}$ of degree 5 which has B\'{e}zier coefficients $\left\{  \widetilde{b}_{i+2} (\mu),\left\vert \mu\right\vert =5\right\}  $, it is then possible to express the coefficients $C_{i}(\beta)$ with respect to $\widetilde{b}_{i+2}(\mu ) $\ for $ i>0 $ in the  modulo 3 congruence.
\begin{proposition}[\textbf{Algo of Subdivision 2}] \label{bbcoeffCi} \textbf{ } \\
For $ i=1,2,3 $, let $\lambda_{i}$ and $\gamma_{i+2}$ be the barycentric coordinates respectively in $\mathcal{K}_{i}$ and $\widetilde{\mathcal{K}}_{i+2}. $\; Let $\widetilde{ \mathcal{P}} _{i+2} \in \mathbb{P}_{5} $ be such that $\widetilde{\mathcal{P}}_{i+2 } = \mathcal{P}_{i}|_{\widetilde {\mathcal{K}}_{i+2}}$ then
\begin{equation} \label{eqnci}%
C_{i}(\beta)={\sum\limits_{\left\vert \mu\right\vert =5}}
\widetilde{b}_{i+2}^{(2)} \left(  \mu\right) B_{\mu_{1}}^{\beta_{1}} \left(\frac{5}{3}\right)  B_{\mu_{3}}^{\beta_{3}}\left(  \frac{5}{3}\right)
\end{equation}
\end{proposition}
\begin{proof}
$\forall\mathcal{V} \in \mathcal{K}_{i}$,%
\begin{equation*}
\left.
\begin{array}{r c l}
\mathcal{V} &=& \lambda_{i1}A_{i+2}+\lambda_{i2}A_{0}+\lambda
_{i3}A_{i+1} , \\
\\
  &=&  \lambda_{i1}(\frac{5}{3}B_{i+1}-\frac{2}{3}A_{0})+ \lambda_{i2}A_{0} +\lambda_{i3}(\frac{5}{3}B_{i}-\frac{2}{3}A_{0}),\\
\\  
  &=& \frac{5}{3}\lambda_{i1}B_{i+1}+(-\frac{2}{3} \lambda_{i1} + \lambda_{i2}-\frac{2}{3}\lambda_{i3})A_{0}+\frac{5}{3} \lambda_{i3} B_{i} \\
  \\
&  = & \gamma_{i+2,1}B_{i+1}+\gamma_{i+2,2}A_{0}+\gamma_{i+2,3}B_{i}  
\end{array}
\right.
\end{equation*}
Let $\mu=\left(  \mu_{1},\mu_{2},\mu_{3}\right)  $  be a mulit-index
of ${\N}$\  such that $\left\vert \mu\right\vert =5,$ %
\begin{equation*}
\left.
\begin{array}{r c l}
B_{\mu}^{5}\left(  \gamma_{i+2}\right) & = & \dfrac{5!}{\mu!}\gamma_{i+2}^{\mu } \\
\\
&  = & \dfrac{5!}{\mu_{1}!\mu_{2}!\mu_{3}!}\left(  \dfrac{5}{3}\lambda_{i1}\right)^{\mu_{1}} \times \left(  \dfrac{5}{3}\lambda_{i3}\right)^{\mu_{3}} \times \left(-\dfrac{2}{3}\lambda_{i1}+\lambda_{i2}-\dfrac{2}{3}%
\lambda_{i3}\right)  ^{\mu_{2}}\\
\\
&  = & \dfrac{5!}{\mu_{1}!\mu_{2}!\mu_{3}!}\left( \dfrac{5}{3}\right)  ^{\mu_{1}}\left( \dfrac{5}{3}\right)^{\mu_{3}} \lambda_{i1}^{\mu_{1}} \lambda_{i3}^{\mu_{3}} \times  \\
\\
&&  {\sum\limits_{\delta+\nu+\eta=\mu_{2}}}\dfrac{\mu_{2}!}{\delta!\nu!\eta!}\left( -\dfrac{2}{3}\lambda_{i1}\right)^{\delta}\left(  \lambda_{i2}\right)^{\nu}\left( -\dfrac{2}{3}\lambda _{i3}\right)  ^{\eta} \\
\\
&  = & {\sum\limits_{\mu_{1}+\delta+\nu+\mu_{3}+\eta=5}}
\dfrac{\beta_{1}!}{\mu_{1}!\delta!}\left( \dfrac{5}{3}\right)^{\mu_{1}} \left( 1-\dfrac{5}{3}\right)^{\delta} \dfrac{\beta_{3}!}{\mu_{3}!\eta!} \times  \\
\\
&& \left(\dfrac{5}{3}\right)^{\mu_{3}}\left(1-\dfrac{5}{3} \right)^{\eta} \times \dfrac{5!}{\beta_{1}!\beta_{2}!\beta_{3}!} \lambda_{i1}^{\beta_{1}}\lambda _{i2}^{\beta_{2}} \lambda_{i3}^{\beta_{3}} \\
\\
&  = & { \sum\limits_{\left\vert \beta\right\vert =5}}
B_{\mu_{1}}^{\beta_{1}}(\dfrac{5}{3})\times B_{\mu_{3}}^{\beta_{3}}(\dfrac{5}{3})\times B_{\beta}^{5}\left( \lambda_{i}\right)
\end{array}
\right.
\end{equation*}
\paragraph{ }  By posing $\beta_{1}=\mu_{1}+\delta$\ , let define
$ \beta_{2}=\nu $\ and $ \beta_{3} =\mu_{3}+\eta $\ with $ \left\vert \beta\right\vert =\left\vert \mu\right\vert =5. $
\paragraph{ }  Then $\left.  \forall\mathcal{V} \in \widetilde{ \mathcal{K}}_{i+2} \right. $
\begin{eqnarray*}
\widetilde{\mathcal{P}}_{i+2}\left( \mathcal{V}\right) &  =  & \sum
\limits_{\left\vert \mu\right\vert =5}\widetilde{b}_{i+2}^{(2)}\left(
\mu\right)  B_{\mu}^{5}\left(  \gamma_{i+2}\right) \\
&  = & \sum\limits_{\left\vert \mu\right\vert =5} \widetilde{b}_{i+2}^{(2)} \left( \mu\right) \left( {\sum \limits_{\left\vert \beta\right\vert =5}}B_{\mu_{1}}^{\beta_{1}}(\frac{5}{3})\times B_{\mu_{3}}^{\beta_{3}}(\frac{5} {3})\times B_{\beta}^{5} \left(  \lambda_{i}\right)\right) \\
&  = & \sum\limits_{\left\vert \beta\right\vert =5}\left({\sum \limits_{\left\vert \mu\right\vert =5}}\widetilde{b}_{i+2}^{(2)} \left( \mu \right)B_{\mu_{1}}^{\beta_{1}}(\frac {5}{3})\times B_{\mu_{3}}^{\beta_{3}}(\frac{5}{3})\right) B_{\beta}^{5} \left(  \lambda_{i}\right) \\
&  = & \sum\limits_{\left\vert \beta\right\vert =5}C_{i}(\beta)\times B_{\beta} ^{5}\left(  \lambda_{i}\right) \\
&  = & \mathcal{P}_{i}\left(  \mathcal{V}\right)
\end{eqnarray*}
If one poses
\begin{equation*}
C_{i}(\beta)={\sum\limits_{\left\vert \mu\right\vert =5}}
\widetilde{b}_{i+2}^{(2)} \left(  \mu\right)  B_{\mu_{1}}^{\beta_{1}}(\frac {5}{3})\times B_{\mu_{3}}^{\beta_{3}}(\frac{5}{3})
\end{equation*}
\end{proof}%
\paragraph{ } At this stage, all B-coefficients defined by equation~\eqref{eqnbbi} are entirely determined using formul\ae \, \eqref{eqnpol3},\,\eqref{eqnelevcoef2}\, and \eqref{eqnci}. The nine (09) B-coefficients which are unknown are thus presented:
\begin{itemize}
\item On $\mathcal{K}_{1}$ we have :%
\begin{align*}
C_{1}(2,3,0)  &  =-\frac{1}{162}\widetilde{b}(3,0,0)+\frac{1}{54} \widetilde{b}(2,1,0)-\frac{1}{54}\widetilde{b}(2,0,1)+\frac{1}{3}\widetilde{b}(1,2,0)+\frac{1}{27}\widetilde{b}(1,1,1)\\
&  -\frac{1}{54}\widetilde{b}(1,0,2)+\frac{25}{81}\widetilde{b}(0,3,0)+\frac{1}{3}\widetilde{b}(0,2,1)+\frac{1}{54}\widetilde{b}(0,1,2)-\frac{1}{1,6,2}\widetilde{b}(0,0,3) \\
C_{1}(1,4,0)  &  =\frac{1}{9}\widetilde{b}(2,1,0)+\frac{2}{9} \widetilde{b}(1,2,0)+\frac{2}{9}\widetilde{b}(1,1,1) \\
&+\frac{1}{9}\widetilde{b}(0,3,0)+\frac{2}{9}\widetilde{b}(0,2,1) +\frac{1}{9}\widetilde{b}(0,1,2)
\end{align*}%
\begin{align*}
C_{1}(1,3,1)  &  =\frac{1}{81}\widetilde{b}(3,0,0)+\frac{17}{54} \widetilde{b}(2,1,0)+\frac{1}{54}\widetilde{b}(2,0,1) +\frac{17}{54}\widetilde{b}(1,2,0)\\
& +\frac{17}{54}\widetilde{b}(1,1,1) + \frac{1}{81}\widetilde{b}(0,3,0)+\frac{1}{54}\widetilde{b}(0,2,1)-\frac{1}{162}\widetilde{b}(0,0,3)\\
C_{1}(041)  &  =\frac{1}{9}\widetilde{b}(3,0,0)+\frac{2}{9} \widetilde{b}(2,1,0)+\frac{2}{9}\widetilde{b}(2,0,1)\\
&  +\frac{1}{9}\widetilde{b}(1,2,0)+\frac{2}{9}\widetilde{b}(1,1,1)+\frac{1}{9}\widetilde{b}(1,0,2)
\end{align*}%
\begin{align*}
C_{1}(0,3,2)  &  =\frac{25}{81}\widetilde{b}(3,0,0)+\frac{1}{3} \widetilde{b}(2,1,0)+\frac{1}{3}\widetilde{b}(2,0,1) +\frac{1}{54}\widetilde{b}(1,2,0)+\frac{1}{27}\widetilde{b}(1,1,1)\\
&  + \frac{1}{54}\widetilde{b}(1,0,2) -\frac{1}{162}\widetilde{b}(0,3,0)-\frac{1}{54}\widetilde{b}(0,2,1) - \frac{1}{54}\widetilde{b}(0,1,2)-\frac{1}{162}\widetilde{b}(0,0,3)
\end{align*}%
\end{itemize}
\begin{itemize}
\item On $\mathcal{K}_{2}$ we have :%
\begin{align*}
C_{2}(2,3,0)  &  =-\frac{1}{162}\widetilde{b}(3,0,0)-\frac{1}{54} \widetilde{b}(2,1,0)+\frac{1}{54}\widetilde{b}(2,0,1) -\frac{1}{54}\widetilde{b}(1,2,0)+\frac{1}{27}\widetilde{b}(1,1,1) \\
& + \frac{1}{3}\widetilde{b}(1,0,2) -\frac{1}{162}\widetilde{b}(0,3,0)+\frac{1}{54}\widetilde{b}(0,2,1) + \frac{1}{3}\widetilde{b}(0,1,2) +\frac{25}{81}\widetilde{b}(0,0,3) \\
C_{2}(1,4,0)  & =\frac{1}{9}\widetilde{b}(2,0,1)+\frac{2}{9} \widetilde{b}(1,1,1)+\frac{2}{9}\widetilde{b}(1,0,2) \\
& +\frac{1}{9} \widetilde{b}(0,2,1)+\frac{2}{9}\widetilde{b}(0,1,2) + \frac{1}{9} \widetilde{b}(0,0,3)\\
C_{2}(1,3,1)  &  =-\frac{1}{162}\widetilde{b}(3,0,0)+ \frac{1}
{54} \widetilde{b}(1,2,0)+\frac{17}{54}\widetilde{b}(1,1,1) +\frac{1}{54}\widetilde{b}(1,0,2)\\
& +\frac{1}{81}\widetilde{b}(0,3,0) + \frac{17}{54}\widetilde{b}(0,2,1) +\frac{17}{54}\widetilde{b}(0,1,2)+\frac{1}{81}\widetilde{b}(0,0,3)
\end{align*}
\end{itemize}
\begin{itemize}
\item On $\mathcal{K}_{3}$ we have:%
\begin{align*}
C_{3}(1,3,1)  &  =\frac{1}{81}\widetilde{b}(3,0,0)+\frac{1}{54}
\widetilde{b}(2,1,0)+\frac{17}{54}\widetilde{b}(2,0,1) +\frac{17}{54}\widetilde{b}(1,1,1)\\
& +\frac{17}{54}\widetilde{b}(1,0,2) - \frac{1}{162}\widetilde{b}(0,3,0)+\frac{1}{54}\widetilde{b}(0,1,2)+\frac{1}{81}\widetilde{b}(0,0,3)
\end{align*}
\end{itemize}
\paragraph{ } By expressing the 10 B-coefficients $ b(\alpha),\, \vert \alpha \vert =3, $\, as a function of the 9 B-coefficients  $ C_{i}(j,2,3-j),$ \\
$ i,j=1,2,3 $\, and $ C_{1}(0,5,0), $\, the 5 unknown B-coefficients on $\mathcal{K}_{1}$\, can be written as follows:
\begin{equation*}
\left\{
\begin{array}{r c l}
C_{1}(2,3,0 )&=& \frac{1}{3}\,C_{{1}} \left( 2,2,1 \right) +\frac{1}{3}\,C_{{2}} \left( 0,2,3 \right) + \frac{1}{3}\,C_{{2}} \left( 1,2,2 \right)   \\
C_{1}(1,4,0 )&=&  C_{{1}} \left( 0,5,0 \right) -\frac{1}{27}\,C_{{1}} \left( 0,2,3 \right) +\frac{1}{9}\,C_{{1}} \left( 2,2,1 \right) -\frac{1}{27} \,C_{{3}} \left( 0,2,3 \right) \\
& & -\frac{1}{9}\,C_{{3}} \left( 1,2,2 \right) -\frac{1}{9}\,C_{{3}} \left( 2,2,1 \right) +{\frac {2}{27}}\,C_{{2}} \left( 0,2,3 \right) +\frac{1}{9}\,C_{{2}} \left( 1,2,2 \right)  \\
C_{1}(1,3,1 )&=& \frac{3}{2}\,C_{{1}} \left( 0,5,0 \right) -\frac{1}{18}\,C_{{1}} \left( 0,2,3 \right) +\frac{1}{6}\,C_{{1}} \left( 1,2,2 \right) +\frac{1}{6}\,C_{{1}} \left( 2,2,1 \right) \\
& & -\frac{1}{18}\,C_{{3}} \left( 0,2,3 \right) -\frac{1}{6}\,C_{{3}} \left( 1,2,2 \right) -\frac{1}{6}\,C_{{3}} \left( 2,2,1 \right) -\frac{1}{18}\,C_{{2}} \left( 0,2,3 \right) \\
& & -\frac{1}{6}\,C_{{2}} \left( 1,2,2 \right) -\frac{1}{6}\,C_{{2}} \left( 2,2,1 \right) \\
C_{1}(0,4,1 )&=& C_{{1}} \left( 0,5,0 \right) +{\frac {2}{27}} \,C_{{1}} \left( 0,2,3 \right) +\frac{1}{9}\,C_{{1}} \left( 1,2,2 \right) -\frac{1}{27}\,C_{{3}} \left( 0,2,3 \right) \\
& & +\frac{1}{9}\,C_{{3}} \left( 2,2,1 \right) -\frac{1}{27}\,C_{{2}} \left( 0,2,3 \right) -\frac{1}{9}\,C_{{2}} \left( 1,2,2 \right) -\frac{1}{9}\,C_{{2}} \left( 2,2,1 \right)  \\
C_{1}(0,3,2 )&=&\frac{1}{3}\,C_{{1}} \left( 0,2,3 \right) +\frac{1}{3}\,C_{{1}} \left( 1,2,2 \right) +\frac{1}{3}\,C_{{3}} \left( 2,2,1 \right)
\end{array}
\right.
\end{equation*}
The same method is applied to obtain the B-coefficients in the triangles $\mathcal{K}_{2} \text{ and } \mathcal{K}_{3}$.
\subsection{Class determination}
\paragraph{ } It just remains to proove that the constructed spline on $ \mathcal{K} $ according to this principle is of $ C^{2}$ class. We just have to verify that for $ i>0 $ and $ 0\leq\alpha\leq2 $,
\[
\partial^{\alpha}\mathcal{P}_{i}=\partial^{\alpha}\mathcal{P}_{i+1}%
\]
along the edge $\left[  A_{0},A_{i+2}\right]  $\ in the modulo 3 congruence, $ i>0$.
\begin{proposition}
Let be $ {\P}_{5}^{2}\left( \mathcal{K}\right)= \mathbb{PP}_{5} \left(  \mathcal{K}\right) \cap C^{2}\left( \mathcal{K}\right) $ the vectorial space of polynomial spline functions of $ C^{2} $ class defined on $ \mathcal{K} $. \\
If a function $ \mathcal{S} $\ is such that its restriction on $\mathcal{K}_{i} $ is a polynomial $\mathcal{P}_{i} $ with B\'{e}zier coefficients which are computed from the \textbf{formula}~\eqref{eqnci} , then $ \mathcal{S} \in{\P}_{5}^{2} \left( \mathcal{K}\right) $.
\end{proposition} %
\begin{proof}
We have to proove that $\mathcal{S}$\ is of $C^{2}$ class along each interior edges of $\mathcal{K}$. \\ %
Recalling \textbf{equation}~\eqref{eqncont} of \textbf{theorem}~\ref{theocont}  with $ \mu=\left(  -1,3,-1\right) $, $ r=2 $, $ d=5 $ and $ 0\leq k\leq2 $, for $ i=1,2,3 $, with respect to the edge $\left[ A_{0},A_{i+2}\right] $.\\
We have to verify for $ 0\leq\rho \leq5-k $ that :%
\begin{equation} \label{eqnderiv1}%
C_{i+1}\left( k,\rho,5-k-\rho\right)= C_{i}^{k}\left( 5-k-\rho
,\rho,0\right) \left( \mu \right)
\end{equation}
in the modulo 3 congruence for $i>0. $ Hence: %
\begin{itemize}
\item for $k=0$ and $0\leq \rho \leq 5$ we have %
\begin{equation} \label{eqncl0}
C_{i+1}\left( 0,\rho ,5-\rho \right) =C_{i}\left( 5-\rho ,\rho ,0\right)
\end{equation}
\item for $k=1$ and $0\leq \rho \leq 4$ we have %
\begin{eqnarray} \label{eqncl1}
C_{i+1}\left( 1,\rho ,4-\rho \right) &=&C_{i}^{1}\left( 4-\rho ,\rho
,0\right) \left(\mu \right)   \\
&=&-C_{i}\left( 5-\rho ,\rho ,0\right) + 3C_{i}\left( 4-\rho ,\rho +1,0\right) -C_{i}\left( 4-\rho ,\rho ,1\right)  \notag
\end{eqnarray}
\item for $k=2$ and $0\leq \rho \leq 3$ we have %
\begin{eqnarray} \label{eqncl2}
C_{i+1}\left( 2,\rho ,3-\rho \right)  &=&C_{i}^{2}\left( 3-\rho ,\rho
,0\right) \left( \mu \right)    \\
&=&C_{i}\left( 5-\rho ,\rho ,0\right) -6C_{i}\left( 4-\rho ,\rho +1,0\right) +9C_{i}\left( 3-\rho ,\rho +2,0\right) \notag  \\ %
&&+2C_{i}\left( 4-\rho ,\rho ,1\right)  -6C_{i}\left( 3-\rho ,\rho +1,1\right) +C_{i}\left( 3-\rho ,\rho ,2\right) \notag
\end{eqnarray}
\end{itemize}
\textbf{Equations}~\eqref{eqncl0}, \eqref{eqncl1} and \eqref{eqncl2} could be verified by using the \textbf{relation}~\eqref{eqnci} of \textbf{proposition}~\ref{bbcoeffCi}.
\end{proof}
\section{Conclusion}
\paragraph{ } This item following the paper entitled "Revisiting the Clough-Tocher $ C^1 $ finite element", gives a new approach for constructing piecewise polynomial finite elements of $ C^{2} $ class based on the subdivision of Clough-Tocher.%
\paragraph{ } Using the subdivision algorithms and the principle of Bernstein-B\'{e}zier's polynomial degree elevating, we show that is possible to compute the B-coefficients of piecewise polynomials of degree 5 defined over B\'{e}zier triangles.%
\paragraph{ } This process will be used in a general way to build a family of finite elements of $C^{r} $ class based on the subdivision of Clough-Tocher for $r\geq2 $. This work is still in progress.
\section*{Acknowledgement}
\paragraph{ } The authors are very grateful and would like to express their thanks to the anonymous referees for their valuable comments and helpful suggestions that improved the present paper.
\bibliography{haudiekouabiblio}

\end{document}